\documentclass[10pt]{amsart}

\usepackage[theoremfont]{newtxtext}
\usepackage[scaled=0.95]{roboto} 
\usepackage[varbb, varvw, smallerops]{newtxmath}
\usepackage[scaled=0.87]{inconsolata} 
\usepackage[cal=cm, calscaled=0.95, frak=euler, frakscaled=0.98]{mathalfa}
\usepackage[T1]{fontenc}
\usepackage{comment}
\usepackage[frac,rfrac,vfrac,root]{mathfixs}

\usepackage{graphicx}
\usepackage{hyperref}
\hypersetup{colorlinks=false}

\title[Pythagorean Triples]{Pythagorean Triples, Complex Numbers, Abelian
Groups and Prime Numbers}
\date{28 January 2021}

\author{Amnon Yekutieli}
\address{Department of  Mathematics,
Ben Gurion University, Be'er Sheva 84105, Israel}
\email{amyekut@math.bgu.ac.il}

\newtheorem{thm}[equation]{Theorem}
\newtheorem{cor}[equation]{Corollary}
\newtheorem{prop}[equation]{Proposition}
\newtheorem{lem}[equation]{Lemma}
\theoremstyle{definition}
\newtheorem{dfn}[equation]{Definition}
\newtheorem{rem}[equation]{Remark}
\newtheorem{exa}[equation]{Example}

\newtheorem{exer}[equation]{Exercise}
\newtheorem{que}[equation]{Question}

\numberwithin{equation}{section}

\newcommand{\iso}{\xrightarrow{%
\smash{\raisebox{-0.5ex}{\ensuremath{\scriptstyle \simeq  \mspace{2mu}}}}}}

\newcommand{\sub}{\subseteq}
\newcommand{\opn}{\operatorname}

\newcommand{\cd}{\mspace{1.3mu}{\cdot}\mspace{1.3mu}}

\newcommand{\rmitem}[1]{\item[\text{\textup{(#1)}}]}

\newcommand{\mbf}[1]{\mathbf{#1}}
\newcommand{\mrm}[1]{\mathrm{#1}}
\newcommand{\mbb}[1]{\mathbb{#1}}

\newcommand{\Ga}{\Gamma}

\newcommand{\ga}{\gamma}
\newcommand{\ep}{\epsilon}
\newcommand{\ze}{\zeta}

\newcommand{\R}{\mathbb{R}}
\newcommand{\Q}{\mathbb{Q}}
\newcommand{\Z}{\mathbb{Z}}

\newcommand{\C}{\mathbb{C}}

\newcommand{\tup}[1]{\textup{#1}}
\newcommand{\bsym}[1]{\boldsymbol{#1}}

\newcommand{\til}[1]{\tilde{#1}}

\renewcommand{\i}{\bsym{i}}

\newcommand{\abs}[1]{\lvert #1 \rvert}

\begin{document}

\begin{abstract}
It is well-known that pythagorean triples can be represented by points of the 
unit circle with rational coordinates. These points form an abelian group, and 
we describe its structure. This structural description yields, almost 
immediately, an enumeration of the normalized pythagorean triples with a 
given hypotenuse, and also to an effective method for producing all such 
triples. This effective method seems to be new. 

This paper is intended for the general mathematical audience, including 
undergraduate mathematics students, and therefore it contains plenty of 
background material, some history and several examples and 
exercises.
\end{abstract}

\maketitle

\tableofcontents


\section{Pythagorean Triples}
 
A {\em pythagorean triple} is a triple $(a, b, c)$ of positive integers
satisfying the equation 
\begin{equation} \label{eqn:1}
a^2 + b^2 = c^2 . 
\end{equation}
The reason for the name is, of course, because of the Pythagoras Theorem, which 
says that the sides of a right angled triangle, with base $a$, height $b$ and  
hypotenuse $c$, satisfy this equation.
See Figure \ref{fig:100}. 

\begin{figure}[ht]
\includegraphics[scale=0.14]{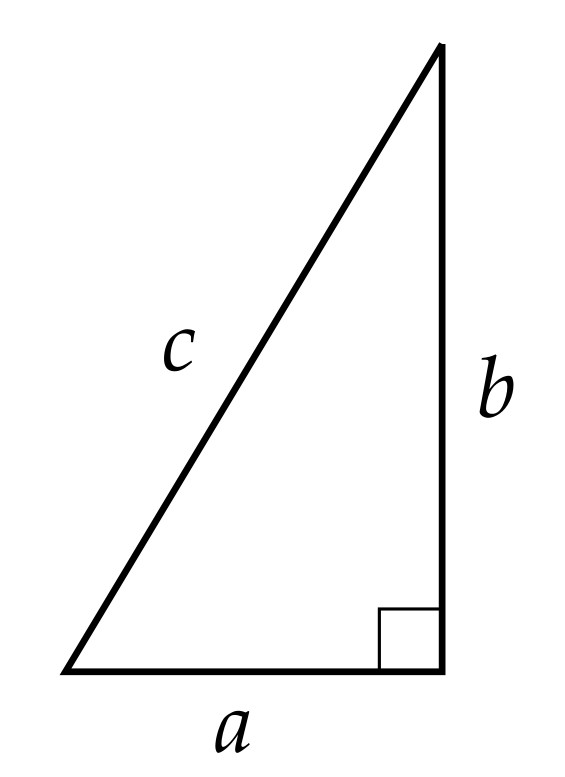}
\caption{Right angled triangle, with base $a$, height $b$ and hypotenuse $c$.}
\label{fig:100}
\end{figure}

We say that two triples $(a, b, c)$ and $(a', b', c')$ are 
{\em equivalent} if the corresponding triangles are similar. 
Numerically this means that there is a positive number $r$, such that  
\[ (a', b', c') = (r \cd a, r \cd b, r \cd c) \]
or 
\[ (a', b', c') = (r \cd b, r \cd a, r \cd c) . \]
See Figure \ref{fig:101}. 
Clearly the number $r$ is rational. 

\begin{figure}[ht]
\includegraphics[scale=0.3]{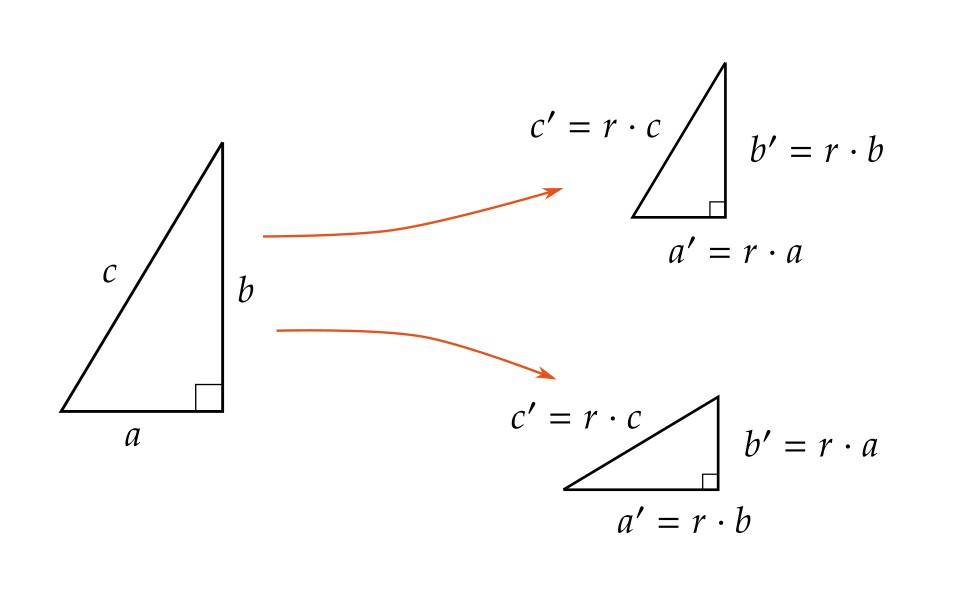}
\caption{Similar right angled triangles.}
\label{fig:101}
\end{figure}

\begin{dfn} \label{dfn:100}
We say that a pythagorean triple $(a, b, c)$ is {\em normalized} if the 
greatest common divisor of these three numbers is $1$, and $a \leq b$. 
\end{dfn}

It is easy to see that every pythagorean triple $(a, b, c)$ is equivalent to 
exactly one normalized triple $(a', b', c')$. 
For this reason we shall be mostly interested in normalized pythagorean triples.
 
\begin{exer} \label{exer:100}
Let $(a, b, c)$ be a normalized pythagorean triple. Show that $c$ is odd, and 
$a < b$. 
\end{exer}

Here is an interesting question:
 
\begin{que} \label{que:100}
Are there infinitely many normalized pythagorean triples?
\end{que}
 
The answer is {\em yes}. This fact was already known to the ancient Greeks. 
There is a formula attributed to Euclid for presenting all pythagorean 
triples, and it proves that there are infinitely many normalized triples.
This formula is somewhat clumsy, and I will not display it. It can be found in 
many sources, including \cite[Chapter 3]{Tk}, or online at \cite{Wi1} or 
\cite{Wo}. Later we will give a geometric argument showing that there are 
infinitely many normalized pythagorean triples. As explained in Remark 
\ref{rem:115}, this geometric argument secretly relies on Euclid's formula.

\begin{dfn} \label{dfn:101}
Let us denote by $\opn{PT}$ the set of all normalized pythagorean triples,
and for each integer $c > 1$ let $\opn{PT}_c$ be the set of normalized 
pythagorean triples with hypotenuse $c$. 
\end{dfn}

Thus we obtain a partition $\opn{PT} = \coprod_{c > 1} \opn{PT}_c$. 
A restatement of Question \ref{que:100} is this: Is the set  $\opn{PT}$
infinite? The next obvious question is:

\begin{que} \label{que:140}
For which $c$ is the set $\opn{PT}_c$ nonempty?
\end{que}

Exercise \ref{exer:100} shows that $\opn{PT}_c = \varnothing$ if $c$ is 
even. Next is a quantitative variant of Question \ref{que:140}. 

\begin{que} \label{que:102} 
What is the size of the set $\opn{PT}_c$~{\hspace{-0.4ex}}?
\end{que}
 
The answer to this question was found in the 19th century, by Gauss. 
We will see it later in the article, in Corollary \ref{cor:31}. 
An even more interesting question is the next one.
 
\begin{que} \label{que:103} 
Given $c$, is there an {\em effective} way to find the elements of 
$\opn{PT}_c$~{\hspace{-0.4ex}}?
\end{que}

An effective method will be presented below, in Theorem \ref{thm:31}.

\section{From Pythagorean Triples to the Rational Unit Circle}
 
It was observed a long time ago that pythagorean triples can be encoded as 
{\em complex numbers on the unit circle}. 
 
Starting from a normalized pythagorean triple 
$(a, b, c)$, we pass to the complex number 
\begin{equation} \label{eqn:110}
z := a + b \cd \i , 
\end{equation}
which has absolute value 
$\abs{z} = \sqrt{a^2 + b^2} = c$. 
Next we introduce the complex number
\begin{equation} \label{eqn:2}
\ze = s + t \cd \i := \frac{z}{\abs{z}} = 
\rfrac{a}{c} + \rfrac{b}{c} \cd \, \i .
\end{equation}
The number $\ze$ has rational coordinates, and it is on the unit circle, in the 
second octant. See Figure \ref{fig:108}.
We can recover the number $z$, and thus the normalized pythagorean triple 
$(a, b, c)$, by clearing the denominators from the pair of rational numbers 
$(s, t) = (\frac{a}{c}, \frac{b}{c})$.
 
\begin{figure}
\includegraphics[scale=0.15]{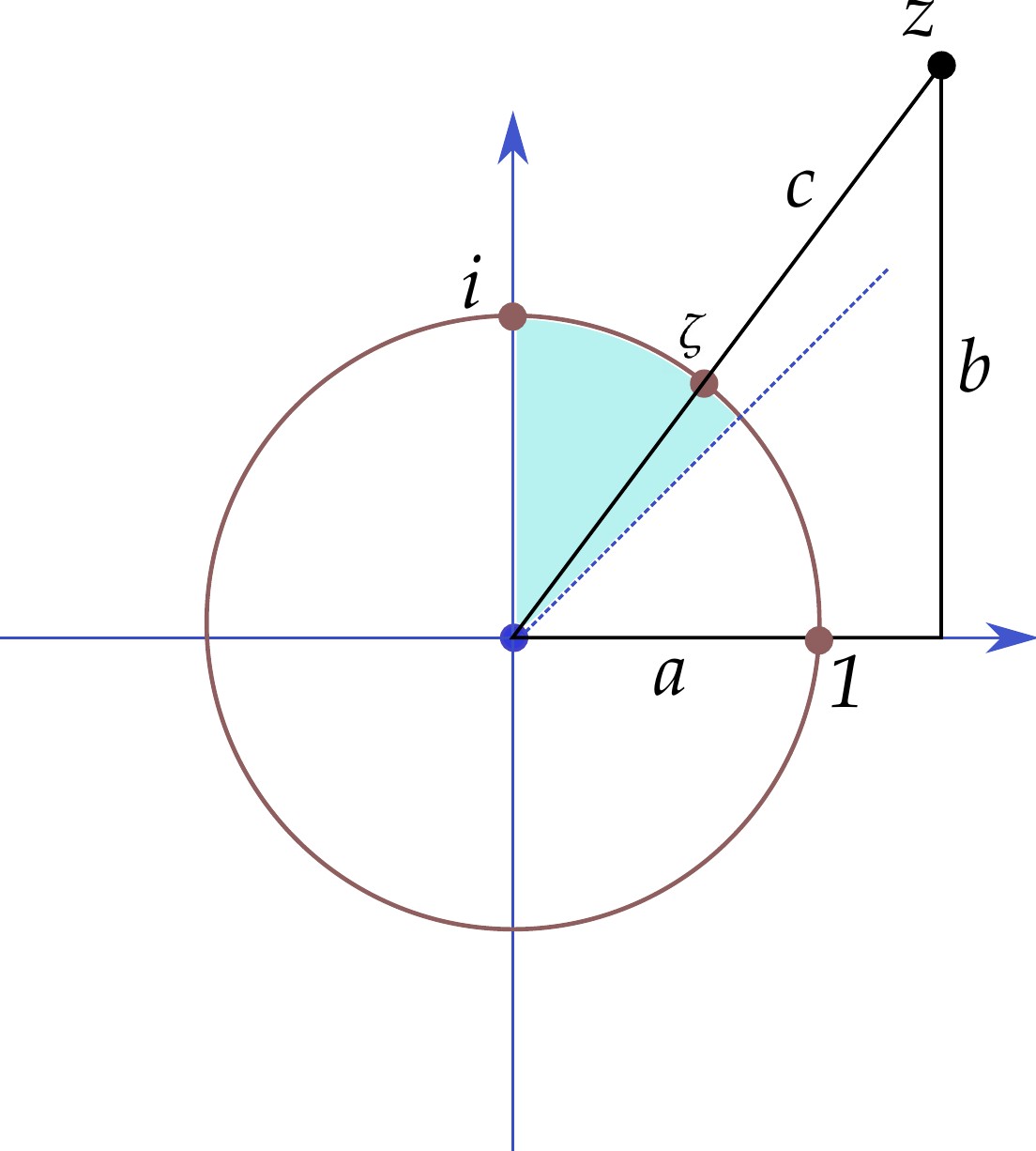}
\caption{The number $\ze$ in the second octant of the unit circle.}
\label{fig:108}
\end{figure}

Actually, there are $8$ different points on the unit circle that encode the 
same pythagorean triple:
\begin{equation} \label{eqn:40}
\pm \ze, \, \pm \i \cd \ze, \, \pm \bar{\ze}, \, \pm \i \cd \bar{\ze} .
\end{equation}
See Figure \ref{fig:103}. 
These points can be obtained from $\ze$ as follows. 
Let $\Ga$ be the group of symmetries of the circle generated by the two 
operations $\ze \mapsto \i \cd \ze$ and $\ze \mapsto \bar{\ze}$.  
This is a nonabelian group of order $8$ (a dihedral group), and we shall call 
it the {\em group of pythagorean symmetries of the 
circle}. The points in (\ref{eqn:40}) are the orbit of the point $\ze$ 
under the action of the group $\Ga$. 
 
\begin{figure}[ht]
\includegraphics[scale=0.15]{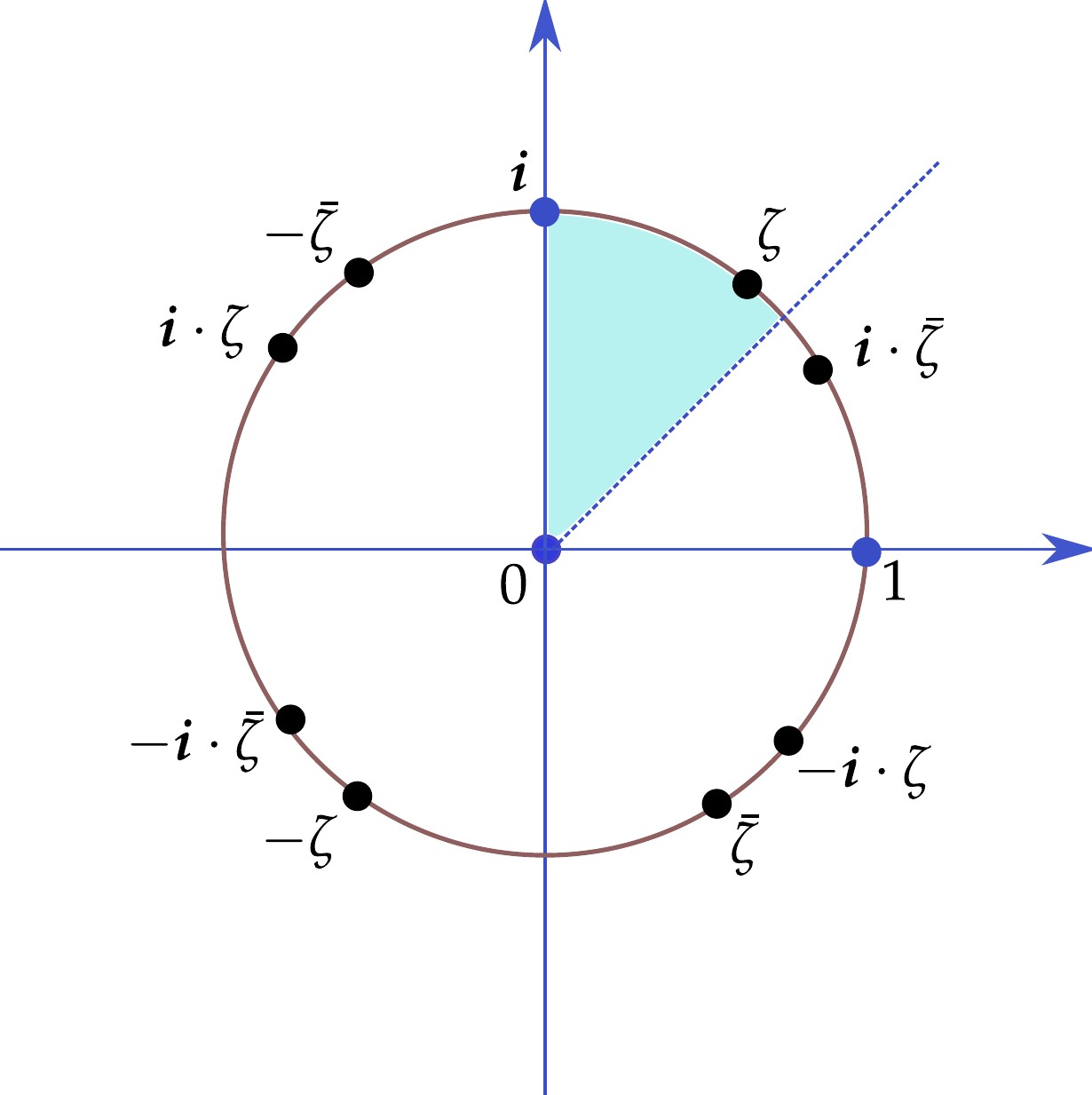}
\caption{The number $\ze$ in the second octant, and its orbit under the action 
of the group $\Ga$ of pythagorean symmetries of the circle.}
\label{fig:103}
\end{figure}

\begin{dfn} \label{dfn:110}
Given a complex number $\ze = s + t \cd \i$ on the unit circle,
with rational coordinates $(s, t)$, and which does not belong to 
$\{ \pm 1, \pm \i \}$, let us denote by $\opn{pt}(\ze)$ the unique 
normalized pythagorean triple $(a, b, c)$ that $\ze$ encodes.
\end{dfn}

In other words, given $\ze$, we first move it to the second octant by an 
element of the group $\Ga$. For $\ze$ in the second octant we have  
$\opn{pt}(\ze) = (a, b, c)$
as in formula (\ref{eqn:2}). 

The function $\opn{pt}$ is a surjection from the set of points
on the unit circle with rational coordinates, excluding the 
four special points $\{ \pm 1, \pm \i \}$, to the set $\opn{PT}$ of normalized 
pythagorean triples. The fibers of the function $\opn{pt}$ are the orbits 
of the group $\Ga$, and each fiber has cardinality $8$. 
 
Let's summarize what we have established so far: 

\begin{prop} \label{prop:115}
The following assertions are equivalent:
\begin{itemize}
\rmitem{i} There are infinitely many normalized pythagorean triples.

\rmitem{ii} There are infinitely many points on the unit circle with rational 
coordinates.
\end{itemize}
\end{prop}

Here is a geometric proof of assertion (ii) in the proposition above. 
Let us denote the unit circle by 
$\mbf{S}^1$. The stereographic projection with focus at $\i$ is the bijective 
function  
$f : \mbf{S}^1 - \{ \i \} \to \mbb{R}$,
which sends the point $\ze \in \mbf{S}^1 - \{ \i \}$ to the unique point 
$f(\ze) \in \R$ that lies on the straight line connecting $\i$ and $\ze$. See 
Figure \ref{fig:110}.
 
\begin{figure}[t]
\includegraphics[scale=0.18]{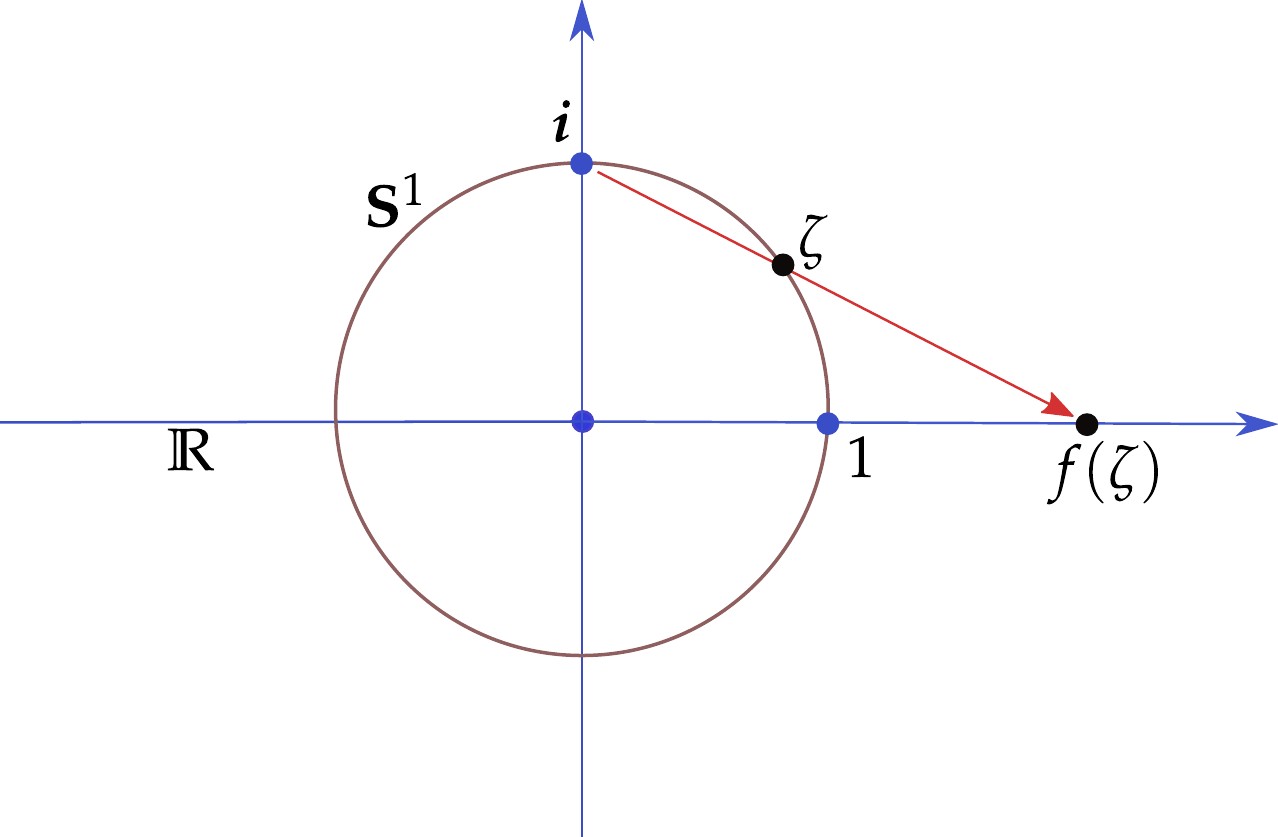}
\caption{The stereographic projection with focus at $\i$.}
\label{fig:110}
\end{figure}

\begin{exer} \label{exer:3}
Show that the point $\ze \in \mbf{S}^1 - \{ \i \}$ has rational coordinates iff 
the number $f(\ze) \in \R$ is rational. (Hint: Use similar triangles.)
\end{exer}
 
Since there are infinitely many rational numbers, we are done. 
This answers Question \ref{que:100} positively. 

\begin{rem} \label{rem:115} 
Here is an algebro-geometric explanation why the stereographic projection $f$ 
sends rational points to rational points. This remark can be safely ignored by 
readers not familiar with the theory of schemes. 

Consider the affine plane over $\Q$, which is the affine scheme
$\mbf{A}^2_{\Q} = \opn{Spec}(\Q[s, t])$. 
Let $X \sub \mbf{A}^2_{\Q}$
be the closed subscheme whose ideal is generated by the polynomial 
$s^2 + t^2 - 1$. The set $X(\R)$ of $\R$-valued points of $X$ is the circle
$\mbf{S}^1$, and the set $X(\Q)$ is precisely the set of points in $\mbf{S}^1$ 
with rational coordinates. 

Let $X' \sub X$ be the open subscheme 
defined by the nonvanishing of the polynomial $t - 1$. 
Let $f : X' \to \mbf{A}^1_{\Q} = \opn{Spec}(\Q[s])$
be the map of affine $\Q$-schemes 
whose formula is $f^*(s) := s / (1 - t)$. 
On $\R$-valued points the function 
\[ f : X'(\R) = \mbf{S}^1 - \{ \i \} \to \mbf{A}^1(\R) = \R \]
is the stereographic projection. 
It turns out that $f$ is an isomorphism of $\Q$-schemes, and the formula for 
its inverse $g : \mbf{A}^1_{\Q} \to X'$ involves the classical expressions of 
Euclid. See \cite{Wi1}. 
Since $f : X' \to \mbf{A}^1_{\Q}$ is an isomorphism of $\Q$-schemes, it 
induces a bijection $f : X'(\Q) \to \mbf{A}^1(\Q) \cong \Q$.
\end{rem}

\section{The Rational Unit Circle as an Abelian Group}
 
Previously we used the notation $\mbf{S}^1$ for the unit circle. 
We will now switch to another notation, which is is better suited for our 
purposes. From now on we shall write
\begin{equation} \label{eqn:115}
G(\R) := \mbf{S}^1 = \{ \ze \in \C \, \mid \, \abs{\ze} = 1 \} .
\end{equation}
Note that the set $G(\R)$ is a group whose operation is complex multiplication,
because 
$\abs{\ze_1 \cd \ze_2} =  \abs{\ze_1} \cd \abs{\ze_2}$
and $\abs{\ze^{-1}} = \abs{\ze}^{-1}$.

Let { $G(\Q)$  be the subset of $G(\R)$ consisting of points with 
rational coordinates}; namely 
\begin{equation} \label{eqn:20}
 G(\Q) = \{ \zeta = s + t \cd \i 
\, \mid \, s, t \in \mbb{Q},  \, s^2 + t^2 = 1 \} .
\end{equation}

\begin{exer} \label{exer:4}
Prove that $G(\Q)$ is a subgroup of $G(\R)$.
(Hint: Inspect the formulas for multiplication and inversion of complex 
numbers.)
\end{exer}

In Remark \ref{rem:116} we explain what lies behind this choice of 
new notation. 

Recall that to answer Question \ref{que:100}, namely to show there are 
infinitely many normalized pythagorean triples, it suffices to 
prove that the abelian group $G(\Q)$ is infinite. This is by Proposition 
\ref{prop:115}. 

We first identify all the elements of finite order in the group $G(\Q)$. These 
are the roots of $1$, namely the elements $\ze \in G(\Q)$ satisfying 
$\ze^n = 1$ for some positive integer $n$. 
Algebraic number theory tells us that there are just four of them:
$1, \i, -1, -\i$.
See \cite[Section 11.2]{Ar} or \cite[Section 5.2]{Tk}.  
(All the results from algebraic number theory that we need can be found in these 
books.)

This means that if we take any element 
$\ze \in G(\Q) - \{ \pm 1, \pm \i \}$, 
the cyclic subgroup that $\ze$ generates, namely the set  
$\{ \ze^n \mid  n \in \Z \} \sub G(\Q)$, 
will be infinite!

Let us consider the familiar normalized pythagorean triple
$(3, 4, 5)$. The corresponding element in $G(\Q)$, by formulas 
(\ref{eqn:110}) and (\ref{eqn:2}), is 
\begin{equation} \label{eqn:118}
\ze :=  \rfrac{3}{5} + \rfrac{4}{5} \cd \i ,
\end{equation}
and it does not belong to $\{ \pm 1, \pm \i \}$. 
So this element $\ze$ has infinite order in the group $G(\Q)$.

This provides us with a second way to answer Question \ref{que:100} 
affirmatively. But we also get, almost for free, a whole list of new normalized 
pythagorean triples! See the first positive powers of $\ze$, and the 
corresponding triples, in Figure \ref{fig:115}. 
The reader is encouraged to verify that these are indeed pythagorean triples. 
It is clear that they are normalized, since the numbers $a_n$ and $b_n$ are not 
divisible by $5$. 
 
\begin{figure}
\renewcommand{\arraystretch}{1.5}
\begin{tabular}[c]{|c|c|c|}
\hline
$n$ & $\zeta^n$ & $\opn{pt}(\zeta^n) = (a_n ,b_n ,c_n)$ \\
\hline \hline
$1$ & $\rfrac{3}{5} + \rfrac{4}{5} \cd \i$ & $(3,4,5)$ \\
\cline{1-3}
$2$ & $-\rfrac{7}{25} + \rfrac{24}{25} \cd \i$ & $(7, 24, 25)$ \\
\cline{1-3}
$3$ & $-\rfrac{117}{125} + \rfrac{44}{125} \cd \i$ & $(44, 117, 125)$ \\
\cline{1-3}
$4$ & $-\rfrac{527}{625}  - \rfrac{336}{625} \cd \i$ & $(336, 527, 625)$ \\
\hline
\end{tabular}
\caption{A list of normalized pythagorean triples.}
\label{fig:115}
\end{figure}

\begin{exer} \label{exer:5}
Find a normalized pythagorean triple with hypotenuse $c = 3125$.
(Later we will see that there is only one.)
\end{exer}
 
\begin{rem} \label{rem:116}
Here is another bit of algebraic geometry, which is not required for 
understanding this paper (yet it did play a role in the discovery of the 
results).

The unit circle, which in Remark \ref{rem:115} was viewed as an affine scheme 
$X$, can also be viewed as an affine group scheme $G$ over $\Q$, namely the 
group $G := \opn{SO}_2$. 
The equality $G(\R) = \mbf{S}^1$ recovers the group structure of the circle. 
Looking at the situation this way, it is clear that 
$G(\Q)$ is a subgroup of $G(\R)$. 
\end{rem}

\begin{rem} \label{rem:117}
After giving a colloquium talk on this topic a few years ago, I was informed 
that the connection between pythagorean triples and the group 
$G(\Q)$ was already observed by O. Tausski \cite{Ts} in 1970. An 
inspection of that paper shows that such a connection was made; yet a full 
understanding of the situation seems to be absent from that paper. In 
particular, the paper \cite{Ts} does not touch the question of enumeration 
of pythagorean triples (cf.\ Corollary \ref{cor:31}), nor does it give an 
effective method for their computation (cf.\ Theorem \ref{thm:31} below). Even a 
list such as in Figure \ref{fig:115} does not appear in that paper. 
\end{rem}

\section{Prime Numbers and the Abelian Group Structure}

We have seen that the pythagorean triples are represented by the 
points of the unit circle with rational coordinates, that we now 
denote by $G(\Q)$. The main result of this section is Theorem \ref{thm:135},
which describes the structure of the abelian group $G(\Q)$ in terms of 
prime numbers. Getting there requires a few steps,
including an understanding of the irreducible elements of the {\em ring of 
Gauss integers}
\begin{equation} \label{eqn:127}
\Z[\i] := \{ m + n \cd \i \, \mid \, m, n \in \Z \} \sub \C  . 
\end{equation}
All facts we use here can be found in \cite[Section 11.5]{Ar} and 
\cite[Section 5.4]{Tk}. 

When we talk about {\em prime integers} we mean {\em positive} prime integers. 
Let's introduce some notation. The set of prime integers is denoted 
by $P$. It is partitioned into 
\begin{equation} \label{eqn:140}
P = P_1 \sqcup P_2 \sqcup P_3 , 
\end{equation}
where 
\begin{equation} \label{eqn:141}
P_i := \bigl\{ p \in P \mid p \equiv i \ (\opn{mod} 4) \bigr\} . 
\end{equation}
Explicitly, 
$P = \{ 2, 3, 5, \ldots \}$, 
$P_1 = \{ 5, 13, 17, \ldots \}$,
$P_2 = \{ 2 \}$ and $P_3 = \{ 3, 7, 11, \ldots \}$.
It is known that the sets $P_1$ and $P_3$ are infinite, but this fact is 
not important for us. 

Our first step is to study the prime numbers $p \in P_1$. 
It turns out that such a prime $p$ can be written as a sum of two 
squares of integers: 
\begin{equation} \label{eqn:122}
p = m^2 + n^2 .
\end{equation}
Because $p$ is odd, we must have $\abs{m} \neq \abs{n}$. 
Therefore, without loss of generality, we can assume that 
$0 < m < n$. Let us define the complex number 
\begin{equation} \label{eqn:123}
q := m + n \cd \i  . 
\end{equation}
Its conjugate is then 
\begin{equation} \label{eqn:126}
\bar{q} = m - n \cd \i . 
\end{equation}
Their product is 
\begin{equation} \label{eqn:124}
q \cd \bar{q} = (m + n \cd \i) \cd (m - n \cd \i) = 
m^2 + n^2  = p .
\end{equation}
We shall be interested in the quotient 
\begin{equation} \label{eqn:125}
\ze_p := q / \bar{q} \in \C . 
\end{equation}
An easy calculation shows that $\ze_p$ has rational coordinates and absolute 
value $1$; thus $\ze_p \in G(\Q)$. 

It might appear that the numbers $q$ and $\bar{q}$ above depend on our choice 
of $m$ and $n$. However, by the classification of the irreducible elements of 
the ring $\Z[\i]$, which we are going to recall below, it follows that there is 
exactly one irreducible divisor $q$ of $p$ in the ring $\Z[\i]$ that sits in 
the second octant, and this is the number $q$ in formula (\ref{eqn:123}). 

\begin{exa} \label{exa:125}
Take the prime $p = 5$. It satisfies 
$5 = 1^2 + 2^2$, so by our convention above we have 
$m = 1, n = 2$, $q = 1 + 2 \cd \i$ and $\bar{q} = 1 - 2 \cd \i$. 
The resulting element of $G(\Q)$ is 
$\ze_5 = q / \bar{q} = - \rfrac{3}{5} + \rfrac{4}{5} \cd \i$.

Note that $\ze_5$ does not coincide with the number 
$\ze = \rfrac{3}{5} + \rfrac{4}{5} \cd \i$
from equation (\ref{eqn:118}). However, they are in the same orbit under the 
action of the group $\Ga$ of pythagorean symmetries of the circle: 
$\ze_5 = - \bar{\ze}$. So they represent the same normalized pythagorean 
triple, which is $(3, 4, 5)$. 
\end{exa}

Our next step is to study some properties of the ring $\Z[\i]$. 
It is known that this ring is a {\em principal ideal domain}, and 
therefore it is a {\em unique factorization domain}.
Let $U$ be the group of invertible elements of $\Z[\i]$, and let 
$Q$ be a complete set of irreducible elements of $\Z[\i]$; to be precise, we 
choose one representative $q \in Q$ from every coset $U \cd q$ of irreducible 
elements. Every nonzero element $z \in \Z[\i]$ has a unique factorization
\begin{equation} \label{eqn:120}
z = u \cd \prod_{i = 1, \ldots, k} \, q_i^{e_i} 
\end{equation}
where $u \in U$, $k \geq 0$, $(q_1, \ldots, q_k)$
is a sequence of distinct elements of $Q$, and the multiplicities are
$e_i \geq 1$. The uniqueness of the factorization (\ref{eqn:120}) is up to a 
permutation of the sequence $(1, \ldots, k)$.

The group of invertible elements of $\Z[\i]$ is 
\begin{equation} \label{eqn:121}
U = \{ 1, \i, -1, -\i \} . 
\end{equation}
The irreducible elements of $\Z[\i]$ are of three types, according to the 
partition (\ref{eqn:140}) of $P$. 
\begin{itemize}
\rmitem{$P_1$} For every prime integer $p \in P_1$, the numbers $q$ and 
$\bar{q}$ from formulas (\ref{eqn:123}) and (\ref{eqn:126}) are irreducible in 
$\Z[\i]$, and they are not equivalent, namely $U \cd q \neq U \cd \bar{q}$. 

\rmitem{$P_2$} The number $1 + \i$ is irreducible in $\Z[\i]$.

\rmitem{$P_3$} Every prime integer $p \in P_3$ is irreducible in $\Z[\i]$.
\end{itemize}

In the next definition we are going to select a particular set of 
representatives $Q$ of the irreducible elements of $\Z[\i]$, according to the 
three types above. 

\begin{dfn} \label{dfn:142}
Define the complete set of irreducible elements $Q$ of $\Z[\i]$ to be
\[ Q := Q_1 \sqcup \bar{Q}_1 \sqcup Q_2 \sqcup Q_3  \]
where:
\begin{enumerate}
\item For every $p \in P_1$, the element $q$ from (\ref{eqn:123})
associated to $p$ belongs to $Q_1$, and the element $\bar{q}$ from 
(\ref{eqn:126}) belongs to $\bar{Q}_1$.  These are all the elements in 
$Q_1 \sqcup \bar{Q}_1$.

\item $Q_2 := \{ 1 + \i \}$. 

\item $Q_3 := P_3$. 
\end{enumerate}
\end{dfn}

Note that the functions $p \mapsto q$ and $p \mapsto \bar{q}$,
from formulas (\ref{eqn:123}) and (\ref{eqn:126}) respectively, 
are bijections $P_1 \iso Q_1$ and $P_1 \iso \bar{Q}_1$. 

It will be important to know the absolute values of the elements of $Q$. 
An element $q \in Q_1$, and its conjugate 
$\bar{q} \in \bar{Q}_1$, have 
$\abs{q} = \abs{\bar{q}} = \sqrt{p}$, where $p = q \cd \bar{q} \in P_1$. 
The element $q = 1 + \i \in Q_2$ has $\abs{q} = \sqrt{2}$.
And an element $q = p \in Q_3 = P_3$ has $\abs{q} = p$.  

The field of fractions of $\Z[\i]$ is 
\begin{equation} \label{eqn:128}
\Q[\i] := \{ s + t \cd \i \, \mid \, s, t \in \Q \} \sub \C  . 
\end{equation}
Like (\ref{eqn:120}), every nonzero element $z \in \Q[\i]$ has a unique 
factorization
\begin{equation} \label{eqn:129}
z = u \cd \prod_{i = 1, \ldots, k} \, q_i^{e_i} 
\end{equation}
where $u \in U$, $k \geq 0$, $(q_1, \ldots, q_k)$
is a sequence of distinct elements of $Q$, but now the multiplicities $e_i$ are 
nonzero integers. 

As shown in formula (\ref{eqn:125}), to each $p \in P_1$ we assign a number 
$\ze_p \in G(\Q)$. In this way we obtain a collection 
$\{ \ze_p \}_{p \in P_1}$ of elements of $G(\Q)$. 

\begin{thm} \label{thm:135}
The abelian group $G(\Q)$ decomposes into a product 
\[ G(\Q) = U \times F , \]
where $U = \{ \pm 1, \pm \i \}$, and $F$ is a free abelian group with basis the 
collection $\{ \ze_{p} \}_{p \in P_1}$. 
\end{thm}

\begin{proof}
We begin by noting that $G(\Q) = \Q[\i] \cap G(\R)$, or in other words
\[ G(\Q) = \{ z \in \Q[\i] \mid \abs{z} = 1 \} . \]

Take an element $z \in G(\Q)$. Being a nonzero element of $\Q[\i]$, it has its 
unique factorization (\ref{eqn:129}). It will be useful for us to alter the 
factorization (\ref{eqn:129}), by inserting more factors from the set $Q$ of 
irreducible elements, with multiplicities $0$, and then to rearrange to 
product. 
The new factorization is this:
\begin{equation} \label{eqn:137}
z = u \cd (1 + \i)^c \cd r_1^{d_1} \cdots r_l^{d_l} \cd 
(q_1^{e_1} \cd \bar{q}_1^{\, e'_1}) \cdots 
(q_k^{e_k} \cd \bar{q}_k^{\, e'_k}) .
\end{equation}
Here $u \in U$; $r_1, \ldots, r_l$ are distinct elements of $Q_3$; 
$q_1, \ldots, q_k$ are distinct elements of $Q_1$; 
and $\bar{q}_1, \ldots, \bar{q}_k \in \bar{Q}_1$ are the conjugates of the 
$q_i$, in the same order.  
The multiplicities 
$c, d_1, \ldots, d_l, e_1, \ldots, e_k, e'_1, \ldots, e'_k$ 
are allowed to be $0$.  

We now examine what the condition $\abs{z} = 1$ imposes on the factorization 
(\ref{eqn:137}). Recall that $\abs{u} = 1$, $\abs{1 + \i} = \sqrt{2}$,
$\abs{r_i} = r_i$, and 
$\abs{q_i} = \abs{\bar{q}_i} = \sqrt{p_i}$, where
$p_i = q_i \cd \bar{q}_i$. It is better to work with $z^2$. We get:
\begin{equation} \label{eqn:138}
1 = \abs{z^2} = 1 \cd 2^c \cd 
r_1^{2 \cd d_1} \cdots r_l^{2 \cd d_l} \cd 
(p_1^{e_1} \cd p_1^{e'_1}) \cdots (p_k^{e_k} \cd p_k^{e'_k}) . 
\end{equation}
Because the integer primes 
$2, r_1, \ldots, r_l, p_1, \ldots, p_k$
are all distinct, we conclude that $c = 0$, $d_i = 0$ and 
$e'_i = -e_i$. This means that in the product (\ref{eqn:137}) we can erase all 
the factors that belong to $Q_2 \cup Q_3$, and also all the  
factors $q_i$ and $\bar{q}_i$ that belong to $Q_1 \cup \bar{Q}_1$ whose 
multiplicities are 
$e_i = e'_i = 0$. Next, for every $i$ such that $e_i \neq 0$ we have 
$q_i^{e_i} \cd \bar{q}_i^{\, e'_i} = 
q_i^{e_i} \cd \bar{q}_i^{\, -e_i} = \ze_{p_i}^{e_i}$, 
see formula (\ref{eqn:125}). 
Therefore, after renumbering the remaining factors in 
(\ref{eqn:137}), and setting a new value for $k$, we obtain the factorization
\begin{equation} \label{eqn:143}
z = u \cd \ze_{p_1}^{e_1} \cdots \ze_{p_k}^{e_k} 
\end{equation}
with $u \in U$, $k \geq 0$, $p_1, \ldots, p_k$ distinct elements of 
$P_1$, and the multiplicities $e_i$ are nonzero integers. 
The factorization (\ref{eqn:143}) is unique up to a permutation of 
$(1, \ldots, k)$. This establishes the decomposition 
$G(\Q) = U \times F$. 
\end{proof}

\begin{rem} \label{rem:135}
Theorem \ref{thm:135} is related to the easy form of Hilbert's Theorem 
90, which says that every element $\ze \in G(\Q)$ satisfies 
$\ze = z / \bar{z}$ for some $z \in \Q[\i]$. See \cite{Wi2}. 
\end{rem}

\section{Back to Pythagorean Triples}

Recall that $G(\Q)$ is the group of points on the unit circle with rational 
coordinates, $U = \{ \pm 1, \pm \i \} \sub G(\Q)$,
and for every point 
$\ze \in G(\Q) - U$ we write $\opn{pt}(\ze)$ for the 
corresponding normalized pythagorean triple. The set of integer 
primes congruent to $1$ modulo $4$ is denoted by $P_1$. For each 
$p \in P_1$ we assigned the element $\ze_p \in G(\Q)$, see (\ref{eqn:125}).
 
\begin{lem} \label{lem:135}
Let $k$ be a positive integer, let $p_1, \ldots, p_k$ be distinct primes in 
$P_1$, let $n_1, \ldots, n_k$ be nonzero integers, let 
$\ep_1, \ldots, \ep_k \in \{ \pm 1 \}$, and let 
\[ (a, b, c) := 
\opn{pt} \bigl( \ze_{p_1}^{\ep_1 \cd n_1} \cdots 
\ze_{p_k}^{\ep_k \cd  n_k} \bigr) \in \opn{PT} . \]
Then $c = p_1^{n_1} \cdots p_k^{n_k}$.
\end{lem}

\begin{proof}
Let's write 
$\ze :=  \ze_{p_1}^{\ep_1 \cd n_1} \cdots \ze_{p_k}^{\ep_k \cd  n_k}$. 
For every $i$ define
\[ \til{q}_i := 
\begin{cases}
q_i & \tup{if} \ \ \ep_i = 1
\\
\bar{q}_i & \tup{if} \ \ \ep_i = -1 . 
\end{cases} \]

Consider the numbers $c' := p_1^{n_1} \cdots p_k^{n_k}$
and $z := c' \cd \ze$. 
For every $i$ we have 
\[ p_i^{n_i} \cd \ze_{p_i}^{\ep_i \cd n_i} = 
(q_i \cd \bar{q}_i)^{n_i} \cd  (q_i / \bar{q}_i)^{\ep_i \cd n_i} = 
\til{q}_i^{\, 2 \cd n_i} , \]
and therefore 
\begin{equation} \label{eqn:150}
z = \til{q}_1^{\, 2 \cd n_1} \cdots \til{q}_k^{\, 2 \cd n_k} .
\end{equation}
We see that $z \in \Z[\i]$, so we can express it uniquely as 
$z = a' + b' \cd \i$ with $a', b' \in \Z$. 
Since $\ze \notin U$ and $z = c' \cd \ze$, it follows that 
$\bar{z} \neq z$ and $\bar{z} \neq -z$, and therefore $a'$ 
and $b'$ are nonzero. Moreover, $\abs{z} = c'$. We conclude that $(a', b', c')$ 
is a pythagorean triple. 

We claim that the triple $(a', b', c')$ is reduced, namely the greatest common 
divisor of these three numbers in $\Z$ is $1$. The prime divisors of $c'$ in 
$\Z$ are $p_1, \ldots, p_k$. Suppose some $p_i$ divides both $a'$ and $b'$.
Then $a' = a'' \cd p_i$ and $b' = b'' \cd p_i$ for some $a'', b'' \in \Z$.  
This will give $z = p_i \cd (a'' + b'' \cd \i)$ in $\Z[\i]$. 
But both irreducibles $q_i$ and $\bar{q}_i$ divide $p_i$ in $\Z[\i]$, and  
this implies that $q_i$ and $\bar{q}_i$ both divide $z$. This is in 
contradiction to the decomposition (\ref{eqn:150}) of $z$ into irreducibles.

At this point we know that either $(a', b', c')$ or $(b', a', c')$ is a 
normalized pythagorean triple, and this is the triple 
$(a, b, c) = \opn{pt}(\ze)$. In any case $c = c'$. 
\end{proof}

\begin{thm} \label{thm:31}  
Let $c$ be an integer greater than $1$, with prime decomposition 
\[ c = p_1^{n_1} \cdots p_k^{n_k}  \]
in $\Z$. Here $p_1, \ldots, p_k$ are distinct prime integers;
$n_1, \ldots, n_k$ are positive integers; and $k$ is a positive integer. 
\begin{enumerate} 
\item If $p_i \equiv 1 \, (\opn{mod} 4)$ for every index $i$, then the function
$\opn{pt}$ restricts to a bijection 
\[ \opn{pt} : 
\bigl\{ \, \ze_{p_1}^{n_1} \cd \ze_{p_2}^{\ep_2 \cd n_2} \cdots 
\ze_{p_k}^{\ep_k \cd n_k} \, \mid \, \ep_2, \ldots, \ep_k \in \{ \pm 1 \} \,
\bigr\} \iso \opn{PT}_c . \]
Here $\ze_{p_i}$ is the number defined in formula (\ref{eqn:125}) 
for the prime $p_i$. 
 
\item Otherwise, the set $\opn{PT}_c$ is empty. 
\end{enumerate}
\end{thm}

\begin{proof}
Consider the set $G(\Q) - U$, the complement of the subgroup $U$ in the group 
$G(\Q)$. We know that the function 
\begin{equation} \label{eqn:165}
\opn{pt} : G(\Q) - U \to \opn{PT} 
\end{equation}
is surjective, and its fibers are orbits of the group $\Ga$ of pythagorean 
symmetries. Thus, passing to the quotient set, we get a bijection 
\begin{equation} \label{eqn:166}
\opn{pt} : \bigl( G(\Q) - U \bigr) / \Ga \iso  \opn{PT} . 
\end{equation}

Let $\Ga_0$ be the subgroup of $\Ga$ of order $2$ generated by the complex 
conjugation $\ga : z \mapsto \bar{z}$. Then $\Ga = \Ga_0 \ltimes U$, a 
semi-direct product. 
According to Theorem \ref{thm:135} there is a group decomposition 
$G(\Q) = U \times F$, where $F$ is a free abelian group with basis the 
collection of elements $\{ \ze_p \}_{p \in P_1}$. The action of the 
group $U$ on the set $G(\Q) - U$ is such that it  induces a bijection 
\begin{equation} \label{eqn:167}
\bigl( G(\Q) - U \bigr) / U \iso F - \{ 1 \} ,
\end{equation}
and this bijection respects the actions of $\Ga_0$. 
Hence we can pass from the bijection (\ref{eqn:166}) to the bijection 
\begin{equation} \label{eqn:151}
\opn{pt} : \bigl( F - \{ 1 \} \bigr) / \Ga_0 \iso  \opn{PT} .
\end{equation}

Now let us fix a positive integer $k$, 
distinct primes  $p_1, \ldots, p_k$ in $P_1$, and  
positive integers $n_1, \ldots, n_k$.
Consider the set 
\[ Z := \bigl\{ \, \ze_{p_1}^{\ep_1 \cd n_1}  \cdots 
\ze_{p_k}^{\ep_k \cd  n_k} \, \mid \ep_1, \ldots, 
\ep_k \in \{ \pm 1 \} \, \bigr\} .  \]
It is a subset of $F - \{ 1 \}$, stable under the action of $\Ga_0$.
Indeed, the conjugation $\ga$ acts by $\ep_i \mapsto -\ep_i$. 
Hence, if we let $Z' \sub Z$ be the subset corresponding to $\ep_1 = 1$, 
the restricted function
$\opn{pt} : Z' \to \opn{PT}$
is injective.
Lemma \ref{lem:135} says that the image $\opn{pt}(Z')$ is contained in 
$\opn{PT}_c$, where 
$c := p_1^{n_1} \cdots p_k^{n_k}$.
The same lemma says that for any element 
$\ze \in F - \{ 1 \}$ that does not belong to $Z$,  $\opn{pt}(\ze)$ is not in 
$\opn{PT}_c$. The conclusion is that the function 
$\opn{pt} : Z' \to \opn{PT}_c$ is bijective. This proves item (1) of the 
theorem. 

As for item (2): Since only numbers $c =  p_1^{n_1} \cdots p_k^{n_k}$
with $p_i \in P_1$ occur as hypotenuses in the image of the bijection 
(\ref{eqn:151}), the subsets $\opn{PT}_c$ are empty for numbers $c$ that are 
not of this kind. 
\end{proof}
 
\begin{cor} \label{cor:31}  
Let $c$ be an integer $> 1$, with prime decomposition as in 
Theorem \ref{thm:31}.
\begin{enumerate} 
\item If $p_i \equiv 1 \, (\opn{mod} 4)$ for every index $i$, then the number 
of normalized pythagorean triples with hypotenuse $c$ is $2^{k - 1}$. 
 
\item Otherwise, there are no normalized pythagorean triples with 
hypotenuse $c$.
\end{enumerate}
\end{cor}

\begin{proof}
(1) The set
\[ Z' =
\bigl\{ \, \ze_{p_1}^{n_1} \cd \ze_{p_2}^{\ep_2 \cd n_2} \cdots 
\ze_{p_k}^{\ep_k \cd n_k} \, \mid \, \ep_2, \ldots, \ep_k \in \{ \pm 1 \} \,
\bigr\}  \]
has cardinality $2^{k - 1}$, and item (1) of the theorem says that the function 
$\opn{pt} : Z' \to \opn{PT}_c$ is bijective. 

\medskip \noindent 
(2) This is clear from item (2) of the theorem.
\end{proof}
 
We end the article with an example and an exercise. 

\begin{exa} \label{exa:140}
Take the number $c = 289$. Its prime factorization is 
$c = 17^2$, and $17 \in P_1$, so by Corollary \ref{cor:31} there is one 
normalized pythagorean triple with hypotenuse $289$. We can say what it quite 
easily. First we express $17$ as a sum of two squares:
$17 = 1 + 16 = 1^2 + 4^2$. We get 
$m = 1$, $n = 4$, $q = 1 + 4 \cd \i$, $\bar{q} = 1 - 4 \cd \i$ and 
\[ \ze_{17} = q / \bar{q} = q^2 / (q \cd \bar{q}) = 
(1 + 4 \cd \i) \cd (1 + 4 \cd \i) \cd \rfrac{1}{17} = 
-\rfrac{15}{17} + \rfrac{8}{17} \cd \i . \]
Next we compute
\[ \ze_{17}^2 = \bigl( -\rfrac{15}{17} + \rfrac{8}{17} \cd \i \bigr) \cd 
\bigl( -\rfrac{15}{17} + \rfrac{8}{17} \cd \i \bigr) = 
\rfrac{161}{289} - \rfrac{240}{289} \cd \i .  \]
This tells us that our normalized pythagorean triple is 
\[ \opn{pt}(\ze_{17}^2) = (161, 240, 289) . \]
We leave it to the reader to verify that this is really a pythagorean triple
(this requires a calculator). Checking for normalization is easy: $17$ does 
not divide $161$ and $240$.
\end{exa}

\begin{exer} \label{exer:140}
Find the two normalized pythagorean triples with hypotenuse $65$. 
\end{exer}

\medskip \noindent 
{\em Acknowledgments}. I wish to thank Eitan Bachmat, Moshe 
Newman, Noam Zimhoni, Moshe Kamenski, Ramin Takloo-Bighash,
Gal Alster, Dor Amzaleg and Steven Miller for their advice in preparing this 
article.

\end{document}